\documentclass[A4paper,12pt]{article}
\usepackage{latexsym}
\usepackage{mathrsfs}
\usepackage{amssymb}
\usepackage{amscd}
\usepackage[dvips]{graphicx}                  
\newtheorem{definition}{Definition}
\newtheorem{theorem}{Theorem}
\newtheorem{corollary}{Corollary}
\newtheorem{lemma}{Lemma}

\newtheorem{proposition}{Proposition}
\newenvironment{proof}[1][Proof]{\textbf{#1.} }{\ \rule{0.5em}{0.5em}}
\date{}
\long\def\symbolfootnote[#1]#2{\begingroup%
\def\thefootnote{$\;$}\footnote[#1]{$^*$#2}\endgroup}
\begin{document}

\title{On Kuratowski partitions}
\author{Ryszard Frankiewicz and Joanna Jureczko}

\maketitle

\begin{abstract}
	In 1935 K. Kuratowski in \cite{KK} posed the problem whether a function $f \colon X \to Y$, ($X$ is completely metrizable and $Y$ is metrizable), with the property that a preimage of each open has the Baire property, is continuous apart from a meager set.
	This paper is a selection of older results related to the question posed by Kuratowski in 1935, coming from among others Solovay and Bukovsk\'y, and quite new ones concerning considerations in Ellentuck topology and some properties of K-ideals, (i.e. ideals associated with Kuratowski partitions).	
\end{abstract}

\section{Introduction and some previous results}

A function $ f \colon X \to Y$ between topological spaces has the Baire property iff for each open set $V \subset Y$ the preimage $f^{-1}(V)$ has the Baire property in $X$ (i.e., is open modulo a meager set in $X$). According to a well-known theorem attributed to K. Kuratowski, if a function $ f \colon X \to Y$ from a metrizable space $X$ to a separable metrizable space $Y$ has the Baire property, then for some meager set $M \subset X$ the restriction
$f$ to $X \setminus M$ is continuous.

In 1935 K. Kuratowski \cite{KK} raised a question whether the assumption of separability of $Y$ is essential. The question makes sense if we assume that $X$ fulfills Baire theorem.
Thus the Kuratowski's question proved to be equivalent to the existence of a partition of a space into meager subsets such that the union of each its subfamily has the Baire property. 

\begin{definition}
	Let $(X, \tau)$ be a topological space and let $\mathcal{F}$ be a partition of $X$ into meager sets. We say that $\mathcal{F}$ is \textit{a Kuratowski partition} if $\bigcup \mathcal{F}'$ has the Baire property for any subfamily $\mathcal{F}' \subset \mathcal{F}$.
\end{definition}

In \cite{S} Solovay considered partitions of the interval $[0,1]$ into meager sets and using metamathematical methods showed that the union of suitable sets of the partition is non-measurable in the sense of category.
He showed that one cannot prove in the Zermelo-Fraenkel set theory ZF (which does not contain the axiom of choice) the existence of a set of reals which is not Lebesgue measurable, more precisely he proved the following theorem.

\begin{theorem}[\cite{S}]
	ZF is consistent with the conjunction of the following statements:
	\\(1) The principle of dependent choices;
	\\(2) every set A of reals is Lebesgue measurable;
	\\(3) every set A of reals has the property of Baire; \\(4) every uncountable set A of reals contains a perfect subset;
	\\(5) let $A = \{A_x: x \in R\}$ be an indexed family of non-empty sets of reals with the reals as the index set, then there are Borel functions $h_1$ and $h_2$ mapping $R$ into $R$ such that
	\\(a) $\{x:h_1(x) \not \in A_x\}$ has Lebesgue measure zero,
	\\(b) $\{x:h_2(x) \not \in A_x\}$ is of first category.
\end{theorem}

In 1979 Bukovsky in \cite{B} using models of set theory, namely a generic ultrapower, proved the following theorem, firstly proved by Solovay (unplished result). The proof of Solovay's result was longer and more complicated than Bukovsk\'y's.	

\begin{theorem}[\cite{B}]
	Let $\{A_\xi \colon \xi \in S\}$ be a partition of the unit $[0,1]$ into pairwise disjoint sets of Lebesgue measure zero. Then there exists a set $S_0 \subseteq S$ such that the union $\bigcup_{\xi\in S_0} A_\xi$ is not Lebesgue measurable.
\end{theorem}

In \cite{EFK} the authors considered partitions of \v Cech complete space.
We remind that a space $X$ is \textit{\v Cech complete} if $X$ is a dense $G_\delta$ subset of a compact space.
Among others they proved the following results.

\begin{theorem}[\cite{EFK}]
	Let $X$ be a \v Cech complete space and has a pseudobase of cardinality $\leq 2^\omega$.
	Let $\mathcal{F}$ be a partition of $X$ into meager sets. Then there exists a family $\mathcal{A} \subset \mathcal{F}$ such that $\bigcup \mathcal{A}$ has not the Baire property.
\end{theorem}

\begin{corollary}[\cite{EFK}]
	If $\mathcal{F}$ is a partition of $R$ into sets of measure zero, then there exists a family $\mathcal{A} \subset \mathcal{F}$ such that $\bigcup \mathcal{A}$ has not the Baire property.
\end{corollary}

Since the questions on Kuratowski's theorem and continuity is the same question, in \cite{EFK1} the authors proved equivalent results to previous theorem but concerning continuous functions. This result is equivalent to one in Selectors Theory.

\begin{theorem}[\cite{EFK1}]
	If $X$ is a \v Cech complete space and has a pseudobase of cardinality $\leq 2^\omega$, then for each map $f \colon X \to Y$ having the Baire property into a space $Y$ with a $\sigma$-disjoint base there exists a meager set $F \subset X$ such that the restriction $f$ to $X \setminus F$ is continuous.
\end{theorem}

Remind that a family $\mathcal{A}$ of subsets of a space $X$ is \textit{point finite} iff for each point $x \in X$ the set $\{A \in \mathcal{A} \colon x \in A \}$ is finite. In \cite{FG} the authors considering  point finite families of meager sets proved the following result. 		

\begin{theorem}[\cite{FG}]
	Let $\mathcal{F}$ be a point finite family of meager sets covering pseudobasically compact space $X$ and $\pi w (X) \leq 2^\omega$. Then there exists a family $\mathcal{A} \subset \mathcal{F}$ such that $\bigcup \mathcal{A}$ has not the Baire property.
\end{theorem}

\section{The equistence of non-measurable sets}

In the literature there are well-known constructions of non-measurable sets like Vitali's, Bernstein's, Sierpi\'nski's and Shelah's (for the last one see \cite{S}).  However, the next theorem is an old result published by Lusin in 1912, (see Theorem 8.2. p. 37 in \cite{JO}).	

\begin{theorem}[\cite{JO}]
	A real-valued function on $R$ is measurable if and only if for every $\varepsilon >0$ there exists a set $E \subset R$ with $\mu(E)<\varepsilon$ such that
	the restriction of $f$ to $R\setminus E$ is continuous, where $\mu$ denotes the Lebesque measure.
\end{theorem}

\noindent
Below we show the combinatorial proof of the existence of one more non-measurable set (see \cite{FJ1}). However, the next result is proved in Measure Theory, the original Lusin Theorem is obtained for Category Theory, (see Appendix in \cite{WS}, p. 172 - 173).
In the proof we will use Fubini Theorem on product measures which we assume to be well-known.

\begin{proposition}[\cite{FJ1}]
	There exists a set $X \subset [0,1]$ of  apositive Lebesgue measure which is non-measurable.
\end{proposition}

\begin{proof}
	Suppose that all subsets of $[0,1]$ are measurable. Let $\mathcal{F}$ be a partition of $X$. Order $\mathcal{F} =  \{F_\alpha \colon \alpha <\lambda\}$, where $\lambda$ is a cardinal not greater than $\mathfrak{c}$. We can assume that $\mu(\bigcup_{\beta < \alpha}F_\beta) = 0$ for all $\alpha < \lambda$ and $\mu(\bigcup \mathcal{F}) >0$,  where $\mu$ denotes a Lebesgue measure.
	
	Consider the following sequences of sets in $[0,1] \times [0,1]$: for $\alpha = 0$ take $F_0 \times F_0$ and for $0 < \alpha < \alpha$ take $F_\alpha \times (\bigcup_{\beta < \alpha} F_\beta)$.
	Then all these sets defined above have Lebesgue measure zero.
	
	For each $F \in \mathcal{F}$ consider open $G_\delta$ sets  $G^{F}_{n}, n \in \omega$ such that $F \subset \bigcap_{n \in \omega} G^{F}_{n}$, $\mu (\bigcap_{n\in \omega}G^{F}_{n}) = 0$ and $\mu (G^{F}_{n}) < \frac{1}{n}$.
	
	Let $\{U_m \colon m \in \omega\}$ be a base in $X$.
	For $n, m \in \omega$ consider
	$$A_{n, m} = \{\alpha < \lambda \colon U_m \subset G^{F_\alpha}_{n}\}.$$
	Since $\mu(F_\alpha) = 0$, for all $\alpha < \lambda$ and $\mu (\bigcup_{\alpha < \lambda}F_\alpha) >0$. At the presence of Fubini Theorem we have that the set
	$\bigcup \{F_\alpha \colon \alpha \in A_{n, m}\}$ is non-measurable for some $n, m \in \omega$.
\end{proof}

In \cite{FJ1} it is also shown that the existence of Kuratowski partitions are strongly connected with the structure of quotient algebras (see also \cite{GS}). (The analogous result to Theorem 7 one can formulate in Category Theory). Let $LM$ denotes the family of all Lebesgue measurable subsets of an unit and $\triangle$
- the ideal of Lebesgue-null sets

\begin{theorem}[\cite{FJ1}]
	Let $\kappa $ be a regular cardinal and $I_\kappa$ be a $\kappa$-complete ideal. Then $P(\kappa)/I_\kappa$ is not isomorphic to $LM/\triangle$.
\end{theorem}

Since it is provable that the existence of a Kuratowski partition of the Baire (complete) metric space is equiconsistent in ZFC with the existence of a measurable cardinal, the results below concern the equistence of such a cardinal.

Let $(X, \Sigma, \mu)$ be a measure space. A measure $\mu$ is \textit{perfect} iff for every $\Sigma$-measurable function $f \colon X \to R$  and $E \subset R$ such that $f^{-1}(E) \in \Sigma$ there exists a Borel set $B \subset E$ such that $\mu(f^{-1}(E)) = \mu (f^{-1}(B))$.

In \cite{FGPR} the authors proved the following result.

\begin{theorem}[\cite{FGPR}]
	Let $(X, \Sigma, \mu)$ be a space with a perfect measure.
	Let $\mathcal{F}$ be a point cover consisting of sets of measure zero and let $|\mathcal{F}|$ be smaller then the first measurable cardinal. Then there exists $\mathcal{A} \subset \mathcal{F}$ such that $\bigcup \mathcal{A}$ is not measurable and has inner measure zero.
\end{theorem}

In \cite{FK} the authors using metamathemacal methods proved the following result.

\begin{theorem} [\cite{FK}]
	The following theories are equiconsistent:
	\\ (1) ZFC $+ \exists$ measurable cardinal;
	\\ (2) ZFC $+$ there is a complete metric space $X$, a metric space $Y$, and a function $f \colon X \to Y$ having the Baire property such that there is no meager set $F \subseteq X$ for which the restriction $f$ to $X \setminus F$ is continuous;
	\\ (3) ZFC $+$ there is a Baire metric space $X$, a metric $Y$, and a function $f \colon X \to Y$ having the Baire property such that there is no meager set $F \subseteq X$ for which the restriction $f$ to $X \setminus F$ is continuous.	
\end{theorem}

\section{Kuratowski partitions in the Ellentuck structure}

A space $[\omega]^\omega$ of infinite subsets of $\omega$, equipped with the topology generated by the base consisting of the sets of the form
$[a, A] = \{B \in [A]^\omega \colon a \subset B \subseteq a \cup A\}$, where $a \in [\omega]^{<\omega}, A \in [\omega]^\omega$ is called \textit{Ellentuck space}.  We call sets of the form $[a, A]$ by \textit{Ellentuck sets} (or shortly EL-sets).

In paper \cite{FS} it is proved the following result.
\begin{theorem} [\cite{FS}]
	No non-meager subspace of Ellentuck space admits a  Kuratowski partition.
\end{theorem}

However, in the proof of this paper there is a gap but the theorem remains true.
In \cite{FJ} the result was corrected. 	

A set $Q \subset [\omega]^\omega$ is \textit{Ramsey null} (or for  short $Q$ is RN) if for every $[a, A]$ there exists $[a', A']\subseteq [a, A]$ such that $[a', A'] \cap Q = \emptyset$.
A Ramsey null set means a meager set in Ellentuck topology.
\\An Ellentuck set is \textit{large}  if is not RN.
\\The next theorem is proved in \cite{FJ}.

\begin{theorem} [\cite{FJ}]
	Let $[a, A]$ be a large EL-set. Let $\mathcal{F}$ be a partition of $[a, A]$ into pairwise disjoint RN-sets such that $\bigcup \mathcal{F}$ is large. Then $\mathcal{F}$ is not Kuratowski's partition.
\end{theorem}

\begin{proof}
	Let $[a, A]$ be a large EL-set.
	Let $\mathcal{F}$ be a partition of $[a, A]$ into RN-sets. Suppose that $\mathcal{F}$ is a Kuratowski partition.
	We will show that $\mathcal{F}_P = \{F \cap P \colon F \in \mathcal{F}\}$, where $P \subset [\omega]^\omega$ is a perfect set, has cardinality continuum.
	We construct inductively a family $\{[a_f, A_f] \colon f \in {^\omega}\{0,1\} \}$ with the following properties: for each $n < \omega$ and $h \in {^n}\{0, 1\}$ there exist $a_{h^\smallfrown \varepsilon} \supseteq a_h$ and $A_{h^\smallfrown \varepsilon} \subseteq A_h$ such that $a_{h^\smallfrown \varepsilon}\setminus a_{h}, \subseteq A$ and $\max a_h < \min A_{h^\smallfrown \varepsilon}$ for any $\varepsilon \in \{0, 1\}$.
	
	Accept the following notation $[a_h, A_h]_{P_h} = [a_h, A_h]\cap P_h$ for $h \in {^n}\{0, 1\}$.
	
	For a given $[a_h, A_h]$ enumerate all subsets of $(a_h\setminus a)\cap A$ by $s_0,..., s_k$.
	We will construct simultaneously large EL-sets
	$[a_{h^\smallfrown 0}, A_{h^\smallfrown 0}]_{P_{h^ \smallfrown 0}}$ and
	$[a_{h^\smallfrown 1}, A_{h^\smallfrown 1}]_{P_{h^ \smallfrown 1}}.$
	\\
	First we will construct the sequences
	$$B^{h}_{0} \supseteq B^{h}_{1}\supseteq .. \supseteq B^{h}_{k}
	\textrm{ and }
	C^{h}_{0} \supseteq C^{h}_{1}\supseteq .. \supseteq C^{h}_{k}$$ as follows: let $B^{h}_{0} = C^{h}_{0} = A_{h}.$
	For given sets $B^{h}_{i}, C^{h}_{j}$ if there exist $B\subseteq B^{h}_{i}$ and $C\subseteq~C^{h}_{j}$ such that there are perfect sets
	$P_{h^\smallfrown 0}, P_{h^\smallfrown 1} \in [\omega]^\omega$ with
	$P_{h^\smallfrown 0} \subseteq [(a_h \cup s_i), B]$,  $P_{h^\smallfrown 1} \subseteq [(a_h \cup s_j), C]$ and  $P_{h^\smallfrown 0} \cap P_{h^\smallfrown 1} = \emptyset$ we put $B^{h}_{i+1} := B$ and $C^{h}_{j+1} := D$. If not then $B^{h}_{i+1} = B^{h}_{i}$ and $C^{h}_{j+1} = C^{h}_{i}$.
	Let $B_{h^\smallfrown 0} = B^{h}_{k} \setminus \{\min B^{h}_{k}\}$
	and
	$C_{h^\smallfrown 0} = C^{h}_{k} \setminus \{\min C^{h}_{k}\}$.
	Now let
	$$a_{h^\smallfrown 0} = (a_h \cup s_i) \cap P_{h^\smallfrown 0}
	\textrm{ and }
	a_{h^\smallfrown 1} = (a_h \cup s_j) \cap P_{h^\smallfrown 1}$$ and
	$A_{h^\smallfrown 0} = B_{h^\smallfrown 0} \cap P_{h^\smallfrown 0}$,
	$A_{h^\smallfrown 1} = B_{h^\smallfrown 1} \cap P_{h^\smallfrown 1}$. Thus the sets $[a_{h^\smallfrown \varepsilon}, A_{h^\smallfrown \varepsilon}]$, $ \varepsilon \in \{0, 1\}$ have been defined.
	
	Let $\bigcap^{\infty}_{n = 0} [a_{f|n}, A_{f|n}]$. By Fusion Lemma, (see e.g. \cite{TJ}), we have that $\bigcup_{f \in {^\omega}\{0, 1\}} \bigcap^{\infty}_{n = 0} [a_{f|n}, A_{f|n}]$ is a perfect set. Denote it by $P$. By the construction, $[a, A]_P$ has cardinality continuum because $[a, A]$ is large.
	
	Let $\mathcal{Q} = \{P_\alpha \in [\omega]^\omega \colon P_\alpha \cap [a, A] \not = \emptyset, \alpha < 2^\omega\}$ be a family of perfect sets  which has non-empty intersection with $[a, A]$.	
	For each $P_\alpha \in \mathcal{Q}$ we choose two distinct sets $D^{0}_{\alpha}, D^{1}_{\alpha} \in [a, A]$ such that
	$$D^{0}_{\alpha}, D^{1}_{\alpha} \in \{D \in [a, A] \colon D \cap P_\alpha \not = \emptyset\} \setminus (\{D^{0}_{\beta} \colon \beta <  \alpha\} \cup \{D^{1}_{\beta} \colon \beta < \alpha\}).$$	
	For any $\varepsilon \in \{0, 1\}$ consider $\{D^{\varepsilon}_{\alpha} \colon \alpha < 2^\omega\}$. There exist $[d^\varepsilon, D^\varepsilon]$ such that $[d^\varepsilon, D^\varepsilon] = \{D^{\varepsilon}_{\alpha} \colon \alpha < 2^\omega\}$. Both such EL-sets are disjoint and their union has not the Baire property. Indeed. If $[d^\varepsilon, D^\varepsilon]$ was a meager set then there would exist $P_\alpha \in \mathcal{Q}$ such that $ [a, A]_{P_\alpha} \subset [\omega]^\omega \setminus [d^\varepsilon. D^\varepsilon], \varepsilon \in \{0, 1\}$. But we have $\{B \in [\omega]^\omega \colon B \cap [d^\varepsilon, D^\varepsilon]_{ P_\alpha} \} \not = \emptyset.$ A contradiction. 	
	If $[d^\varepsilon, D^\varepsilon], \varepsilon \in \{0, 1\}$ is not meager then there exists $P_\alpha \in \mathcal{Q}$ such that $P_\alpha \subset [d^\varepsilon, D^\varepsilon]$  but by the construction,
	$[d^{\varepsilon'}, D^{\varepsilon '}]_{P_\alpha} \not = \emptyset$ for $\varepsilon ' \not = \varepsilon $. A contradiction to $[d^0, D^0], [d^1, D^1]$  being disjoint.
\end{proof}

\section{Kuratowski partitions and K-ideals}

Let $\kappa$ is a regular cardinal.
With any Kuratowski partition
$$\mathcal{F} = \{F_\alpha \colon \alpha < \kappa\}$$
we may associate an ideal
$$I_\mathcal{F} = \{A \subset \kappa \colon \bigcup_{\alpha \in A} F_\alpha \textrm{ is meager }\}$$
which we call \textit{a $K$-ideal}.

Let $S$ be a set of a positive measure. An \textit{$I$-partition} of $S$ is a maximal family $W$ of subsets of $S$ of positive measure such that $A \cap B \in I$ for any distinct $A, B \in W$.
An $I$-partition $W_1$ of $S$ is a \textit{refinement} of an $I$-partition $W_2$ of $S$, $W_1 \leq W_2$, if every $A \in W_1$ is a subset of some $B\in W_2$.
Let $\kappa$ be a regular cardinal  and let $I$ be a $\kappa$-complete ideal on $\kappa$ containing singletons.
The ideal $I$ is \textit{precipitous} if whenever $S$ is a set of  a positive measure and $\{W_n \colon n < \omega\}$ are $I$-partitions of $S$ such that
$$W_0 \geq W_1\geq ... \geq W_n \geq ...$$
then there exists a sequence of sets
$$A_0 \supseteq A_1\supseteq ... \supseteq A_n \supseteq ...$$
such that $A_n \in W_n$ for each $n$, and $\bigcap_{n=0}^{\infty} X_n$ is nonempty, (see also \cite{TJ} p. 438-439).

Let $I^+ = \mathcal{P}(\kappa) \setminus I$.
Consider a set
$$X(I) = \{x \in ^{\omega}(I^+) \colon \bigcap\{x(n) \colon n \in \omega\} \not = \emptyset \textrm{ and } \forall_{n\in \omega} \bigcap\{x(m) \colon m < n\} \in I^+\}.$$
As was pointed out in \cite{FK} the set $X(I)$ is considered as a subset of a complete metric space $(I^+)^\omega$, where $I^+$ is equipped with the discrete topology.
In \cite{FK} the following facts were proved (see \cite{FK}, Proposition 3.1. and Theorem 3.2.).

\begin{proposition}[\cite{FK}]
	$X(I)$ is a Baire space iff $I$ is a precipitous ideal.
\end{proposition}

\begin{theorem}[\cite{FK}]
	Let $I$ a precipitous ideal on some regular cardinal. Then there is a Kuratowski partition of the metric Baire space $X(I)$.
\end{theorem}

The following proposition is proved in \cite{JW}.

\begin{proposition} [\cite{JW}]
	Let $2^\omega = 2^{\omega_1}$.
	Then there exists a metric Baire space having a Kuratowski partition for which a completion does not have a Kuratowski partition.
\end{proposition}

\noindent
Moreover in \cite{JW} it is shown  that a $K$-ideal  is not necessarily precipitous. Before it we show two related results. (See \cite{RE} p. 103 for the definition of a direct sum of spaces).

\begin{proposition}[\cite{JW}]
	Let $\kappa$ be a regular cardinal.
	Let $X$ be a space having a Kuratowski partition of cardinality $\kappa$ and let $\Pi$ be a family of all permutations of $\kappa$. Then the direct sum   $\oplus_{\pi \in \Pi} X_\pi$ has a Kuratowski partition.
\end{proposition}

\noindent
With the Kuratowski partition $\mathcal{F}^*$ of a direct sum $\oplus_{\pi \in \Pi} X_\pi$ of copies of $X$  considered in the proof of Proposition 18 we may assiociate a $K$-ideal $I_{\mathcal{F}^*}$. In Theorem 19 we show that such an ideal is a Frech\'et ideal,  (i.e. $I_{\mathcal{F}} = [\kappa]^{<\kappa}$).

\begin{theorem}[\cite{JW}]
	Let $\kappa$ be a regular cardinal.
	Let $X$ be a space having a Kuratowski partition $\mathcal{F}$ of cardinality $\kappa$. Then a $K$-ideal $I_{\mathcal{F}^*}$ of is a Frech\'et ideal.	
\end{theorem}

\noindent
The following lemma is proved in \cite{TJ}, (Lemma 35.9 p. 440).

\begin{lemma} [\cite{TJ}]
	Let $\kappa$ be a regular uncountable cardinal. The ideal $I = \{X \subseteq \kappa \colon |X| < \kappa\}$ is not precipitous.
\end{lemma}

\noindent
Using Lemma 20 we immediately obtain the next corollary.

\begin{corollary}[\cite{JW}]
	Let $X$ be a space having a Kuratowski partition. Then the $K-$ideal can not to be precipitous.		
\end{corollary}

Let $B(2^\kappa)$ be a Baire space, where $\kappa$ is a cardinal. In \cite{FK} the authors proved the following theorem.

\begin{theorem}[\cite{FK}]
	Assume that $J$ is an $\omega_1$-complete ultrafilter on $\kappa$. Then $B(2^\kappa)$ has a Kuratowski partition of cardinality $\kappa$.
\end{theorem}

\noindent
By Theorem 3.3. and Theorem 3.4. in \cite{FK} the existence of a Kuratowski partition of an arbitrary space is equiconsistent with the existence of a measurable cardinal.
In the next theorem it  is shown that  reducing an ideal up to the Frechet ideal can  be obtained by enlarging the space, as a direct sum. On the other hand, enlarging an ideal depends on localization property, it means reducing the space.  The idea of getting measurable cardinal is a result of localization (see \cite{FK}).

\begin{theorem}[\cite{JW}]
	Let $\kappa$ be a measurable cardinal. Then each $\kappa$-complete ideal $I$ on $\kappa$ can be represented by some $K$-ideal, (i.e. for each $\kappa$-complete ideal $I$ on $\kappa$ there exists a space having a Kuratowski partition $\mathcal{F}$ of cardinality $\kappa$ such that $I$ is of the form $I_\mathcal{F}$).
\end{theorem}

\noindent
If $\kappa$ is a nonmeasurable cardinal but there exists a Kuratowski partition of cardinality $\kappa$ of a space $X$ then one can obtain each $\kappa$-complete ideal $K$ such that $J \subseteq K \subseteq I$, where $J$ is a Frech\'et ideal and $I$ is an ideal from previous theorem.
Thus, the existence of a Kuratowski partition leads to the statement that there exists a precipitous ideal, (see \cite{FJ}).
\\
At the presence of the previous results, in \cite{FJ2} it is shown the next theorem.

\begin{theorem}[\cite{FJ2}]
	Let $X$ be a metric Baire space with a Kuratowski partition $\mathcal{F}$. Then there exists an open set $U \subset X$ such that a $K$-ideal $I_{\mathcal{F}\cap U}$ is precipitous.
\end{theorem}

\begin{proof}
	Let $\kappa = \min \{|\mathcal{F}| \colon \mathcal{F} \textrm{ is a Kuratowski partition of } X\}.$
	Let $\mathcal{F} = \{F_\alpha \colon \alpha < \kappa\}$ be a fixed Kuratowski partition of $X$ and let $I_{\mathcal{F}}$ be a $K$-ideal for $\mathcal{F}$. Consider 
	$FN(\kappa) = \{f \in {^X}\kappa \colon \exists_{\mathcal{U}_f} \textrm{ family of open disjoint sets which is }  \textrm{dense in } X \textrm{ and } \\ \forall_{\alpha \in \kappa} \forall_{U \in \mathcal{U}_f} f \textrm{ is constant on }  F_\alpha\cap U\}$, (compare \cite[proof of Theorem 3.3]{FK}).
	For each $f \in FN(\kappa)$ and each $U \in \mathcal{U}_f$ consider a $K-$ideal $$I_{\mathcal{F}\cap U} = \{A \subset \kappa \colon \bigcup_{\alpha \in A} (F_\alpha \cap U) \textrm{ is meager, } F_\alpha \in \mathcal{F}\}.$$ We claim that there exists $f \in FN(\kappa)$ and $U \in \mathcal{U}_f$ such that $I_{\mathcal{F}\cap U}$ is precipitous. 
	
	Suppose not, then by  Lemma 2.1, for all $f \in FN(\kappa)$ and all $U \in \mathcal{U}_f$ there are sequences of functionals  
	$F^{f}_{0}(U) > F^{f}_{1}(U)> ... $, on $U$ such that $F^{f}_{i}(U) \subseteq FN(\kappa)$, $i = 0, 1, ...$.
	
	Fix $f \in FN(\kappa)$ and $U\in \mathcal{U}_f$. 
	Let $W^{f}_{i}(U)$ be  $I_{\mathcal{F}\cap U}$-partitions for $F^{f}_{i}(U)$,  $i = 0, 1, ...\ $.
	Then  $W^{f}_{0}(U) \geq W^{f}_{1}(U)\geq ...\ $.
	Let $X^{f}_{i}(U) \in W^{f}_{i}(U)$ such that
	$X^{f}_{0}(U) \supseteq X^{f}_{1}(U)\supseteq ...\ $.
	Since $I_{\mathcal{F}\cap U}$ is not precipitous,  $\bigcap_{i=0}^{\infty} X^{f}_{i}(U) = \emptyset$.
	
	Now for each $i = 0, 1, ...$ consider $F^{f}_{i}: =\bigcup_{U \in \mathcal{U}_f} F^{f}_{i}(U).$
	Hence $F^{f}_{i} \subseteq FN(\kappa)$ and 	$F^{f}_{0} > F^{f}_{1} >...$
	is the sequence of functionals on $S$. 
	Let  $W^{f}_{i} = \bigcup_{U \in \mathcal{U}_f} W^{f}_{i}(U)$. Then $W^{f}_{i}$ are $I_{\mathcal{F}}$-partitions for $F^{f}_{i}$, $i = 0, 1, ...$ and $W^{f}_{0} \geq W^{f}_{1} \geq ...\ $.
	Let $X^{f}_{i} = \bigcup_{U \in \mathcal{U}_f}X^{f}_{i}(U)$ be elements of $W^{f}_{i}$, $i = 0, 1, ...$  with $X^{f}_{0} \supseteq X^{f}_{1} \supseteq ...\ $. Let $h^{f}_{i} \in F^{f}_{i}$ be functions with domains $X^{f}_{i}$, $i = 0, 1, ...\ $. By Baire Category Theorem there exists $x \in \bigcap_{i=0}^{\infty} X^f_i \subseteq S$ such that $h^{f}_{0}(x) > h^{f}_{1}(x) >...\ $. A contradiction with the well-foundness of the sequence.
\end{proof}

In the light of consideration on Kuratowski partitions, the newest result of this topic given in \cite{FJ3} is the following theorem.

\begin{theorem}[\cite{FJ3}]
	If there is a metric Baire space which admits a Kuratowski partition then there is a measurable cardinal.
\end{theorem}

{\sc Ryszard Frankiewicz}
\\
Institute of Mathematics, Polish Academy of Sciences, Warsaw, Poland
\\
{\sl e-mail: rf@impan.pl}
\\

{\sc Joanna Jureczko}
\\
Cardinal Stefan Wyszy\'nski University
\\
Faculty of Mathematics and Natural Sciences, Warsaw, Poland
\\
{\sl e-mail: j.jureczko@uksw.edu.pl}

\end{document}